\documentclass[12pt]{amsart}
\usepackage{amsmath,amsthm,amssymb,amscd}
\usepackage{graphicx}
\usepackage[usenames,dvips]{color} 
\usepackage[bookmarks=false,colorlinks,linkcolor=blue,citecolor=green]{hyperref}

\usepackage[all]{xy}
\definecolor{mains}{cmyk}{.3, .85, .75, 0}

\begin{document}
\newcommand{\CC}{\mathcal{C}}
\newcommand{\cA}{{\mathcal A}}
\newcommand{\blambda}{{\boldsymbol{\lambda}}}
\newcommand{\bDelta}{{\boldsymbol{\Delta}}}
\newcommand{\bdelta}{{\boldsymbol{\delta}}}
\newcommand{\cQ}{{\mathcal Q}}
\newcommand{\dQ}{{\mathcal Q}^0}                               
\newcommand{\qQ}{{\widehat{\mathcal Q}}}
\newcommand{\bT}{{\overline{\mathcal T}}}
\newcommand{\T}[1]{{\mathcal T}_#1}                                         
\newcommand{\pT}[1]{{\bf T}_{#1}(\R)}
\newcommand{\BHV}[1]{{\rm BHV}_{#1}}                    
\newcommand{\iBHV}[1]{{\rm BHV}_{#1}^+}                 
\newcommand{\OM} [1] {\ensuremath{{\overline{\mathcal M}}{_{0, #1}^{{\rm or}}(\R)}}}
\newcommand{\KM} [1] {\ensuremath{{\over  line{\mathcal M}}{_{0, #1}^{{\rm kap}}(\R)}}}
\newcommand{\aff}  {{\rm Aff}}                 
\newcommand{\conf}{{\rm Config}}                 

\newcommand{\Con} {{\rm Config}}
\newcommand{\M}  [1] {\ensuremath{{\overline{\mathcal M}}{_{0, #1}(\R)}}}   
\newcommand{\cM} [1] {\ensuremath{{\mathcal M}_{0, #1}}}                    
\newcommand{\CM} [1] {\ensuremath{{\overline{\mathcal M}}{_{0, #1}}}}       
\newcommand{\oM} [1] {\ensuremath{{\mathcal M}_{0, #1}(\R)}}                
\newcommand{\oZ} [1] {\ensuremath{{\mathcal Z}^{#1}}}                       
\newcommand{\PSL} {\Pj\Sl_2(\R)}                           
\newcommand{\PGL} {\Pj\Gl_2(\R)}                           
\newcommand{\PGLC} {\Pj\Gl_2(\Cx)}                         
\newcommand{\CP} {\Cx\Pj^1}                                
\newcommand{\RP} {\R\Pj^1}                                 
\newcommand{\Cx} {{\mathbb C}}                             
\newcommand{\R} {{\mathbb R}}                              
\newcommand{\D} {{\mathcal D}}                             
\newcommand{\Q} {{\mathbb Q}}                              
\newcommand{\N} {{\mathbb N}}                              
\newcommand{\Z} {{\mathbb Z}}                            
\newcommand{\Pj} {{\mathbb P}}                             
\newcommand{\La} {{\mathbb L}}                             
\newcommand{\Sg} {\mathbb S}                               
\newcommand{\SO} {{\rm SO}}                                
\newcommand{\Gl} {{\rm Gl}}                                
\newcommand{\Sl} {{\rm Sl}}                                
\newcommand{\Op}{{\mathcal{O}}}                            
\newcommand{\la}{{\langle}}                                
\newcommand{\ra}{{\rangle}}                                

\newcommand{\J} {\mathcal{J}}

\newcommand{\K} {\mathcal{K}}
\newcommand{\Pg} {\mathcal{P}}
\newcommand{\Cox}{{G}}                                
\newcommand{\PG}[1] {\mathcal{P}G_{#1}}               
\newcommand{\C} [1]  {\mathcal C{#1}}                      
\newcommand{\Cp} [1] {{\Pj\C{}}({#1})}                     
\newcommand{\Cm} [1] {\C{} ({#1})_{\#}}                    
\newcommand{\Cpm} [1] {\Pj{}\Cm{#1}}                       
\newcommand{\Min} [1] {\textrm{Min}(\C{#1})}               

\newcommand{\Sp}{\mathcal S}

\newcommand{\sn}[1]{{\rm CSN}_{#1}}
\newcommand{\n}[1]{\mathfrak{S}_{#1}}
\newcommand{\mc}{\mathcal}
\newcommand{\ms}{\mathcal}
\newcommand{\bd}{overlap space }

\newcommand{\KG} {{\mathcal{K}} G}

\newcommand{\Tr} {{\tau}}                                  
\newcommand{\tri}{{\triangle}}                             

\newcommand{\cB}{{\mathcal B}}
\newcommand{\cC}{{\mathcal C}}

\newcommand{\iDelta}{{\overset{o}\Delta}}
\newcommand{\wsM}{{\widetilde{\sf M}}}
\newcommand{\orb}{{\rm orb}}
\newcommand{\PGl}{{\rm PGl}}
\newcommand{\SP}{{\rm SP}}
\newcommand{\Ot}{{\rm O}}
\newcommand{\reg}{{\rm reg}}
\newcommand{\Conf}{{\rm Config}}                 
\newcommand{\oOM} [1] {\ensuremath{{\mathcal M}_{0,#1}^{{\rm or}}(\R)}}                 
\newcommand{\tsM} [1]{\ensuremath{{\widetilde{\sf M}}{_{0, #1}(\R)}}}
\newcommand{\gM} [1] {{\sf M}_{#1}}
\newcommand{\tgM} [1] {\widetilde{\sf M}_{#1}}

\newcommand{\tQ}{\widetilde{\cQ}}
\newcommand{\sQ}{{\sf Q}}
\newcommand{\tsQ}{{\widetilde{\sf Q}}}
\newcommand{\balpha}{{\boldsymbol{\alpha}}}
\newcommand{\B}{{\bf 1}}
\newcommand{\e}{{\bf e}}
\newcommand{\y}{{\bf r}}
\newcommand{\x}{{\bf x}}
\newcommand{\tx}{{\bar{x}}}
\newcommand{\bn}{{\bf n}}
\newcommand{\metmap}{\phi}
\newcommand{\half}{{\textstyle{\frac{1}{2}}}}

\newcommand{\hide}[1]{}

\newcommand{\minibreak}{
\noindent\rule[0pt]{\textwidth}{1pt}
\noindent\rule[14pt]{\textwidth}{2pt}}


\theoremstyle{plain}
\numberwithin{equation}{section}
\newtheorem{thm}{Theorem}
\newtheorem{prop}[thm]{Proposition}
\newtheorem{cor}[thm]{Corollary}
\newtheorem{lem}[thm]{Lemma}
\newtheorem{conj}[thm]{Conjecture}
\newtheorem{prob}[thm]{Problem}
\newtheorem{ques}[thm]{Question}

\theoremstyle{definition}
\newtheorem{defn}{Definition}
\newtheorem*{exmp}{Example}

\theoremstyle{remark}
\newtheorem*{rem}{Remark}
\newtheorem*{hnote}{Historical Note}
\newtheorem*{nota}{Notation}
\newtheorem*{ack}{Acknowledgments}

\def\ve#1{\mathchoice{\mbox{\boldmath$\displaystyle\bf#1$}}
{\mbox{\boldmath$\textstyle\bf#1$}}
{\mbox{\boldmath$\scriptstyle\bf#1$}}
{\mbox{\boldmath$\scriptscriptstyle\bf#1$}}}
\newcommand{\conv}{\mbox{conv}}
\newcommand{\Tn}{\mathcal T_n}

\newcommand{\squeezelist}{\setlength{\itemsep}{-2pt}}

\title{Split Network Polytopes and Network Spaces}
\author{Satyan L.\ Devadoss, Cassandra Durell, and Stefan Forcey}

\address{Satyan Devadoss, Department of Mathematics, University of San Diego}
\address{Cassandra Durell, Department of Mathematics, University of Akron}
\address{Stefan Forcey, Department of Mathematics, University of Akron}

\begin{abstract}
 Phylogenetics begins with reconstructing biological family trees from genetic data. Since Nature is not limited to tree-like histories, we use  \emph{networks}  to organize our data, and have discovered new polytopes, metric spaces, and simplicial complexes that help us do so. Moreover, we show that the space of phylogenetic trees dually embeds into the \emph{Balanced Minimum Evolution} polytope, and use this result to find a complex of faces within the subtour-elimination facets of the \emph{Symmetric Traveling Salesman} polytope, which is shown to be dual to a quotient complex in network space.
\end{abstract}
\maketitle

\section{Introduction}

A classical problem in computational biology is the inference of a phylogenetic tree from the aligned DNA sequences of $n$ species. One can construct a distance between two species, typically via a model of mutation rates, using the probabilistic calculation pairwise on taxa. Such information can be encoded by
an $n \times n$ real symmetric, nonnegative matrix called a \emph{dissimilarity matrix}, often given as a ${n \choose 2}$-dimensional \textit{distance vector} $\mathbf{d}$ with entries $d_{ij}$ in lexicographic order.  The classical phylogenetic problem is then to reconstruct a tree (possibly with weighted edges) that
represents this matrix. We say that $\mathbf{d}_t$ is \textit{additive} when the entries correspond perfectly to the summed edge values of a weighted tree $t$.

Oftentimes, however, dissimilarity matrices are not additive metrics. This may be due to horizontal gene
transfer, recombination, or gene duplication  \cite{hrs}. Indeed, assuming a tree model and applying phylogenetic inference tools on the data can be misleading \cite{bm}.   In such cases, the underlying evolutionary relationships are better represented by a \emph{split network}, a generalization of a tree in which multiple parallel edges signify divergence; see Figure~\ref{f:network}.  From a biological context, network representations are important  visualization tools at the beginning of data analysis, allowing researchers to see several hypothesized trees simultaneously, while looking for evidence of hybridization or reticulation.

\begin{figure}[h]
\centering{\includegraphics[width=\textwidth]{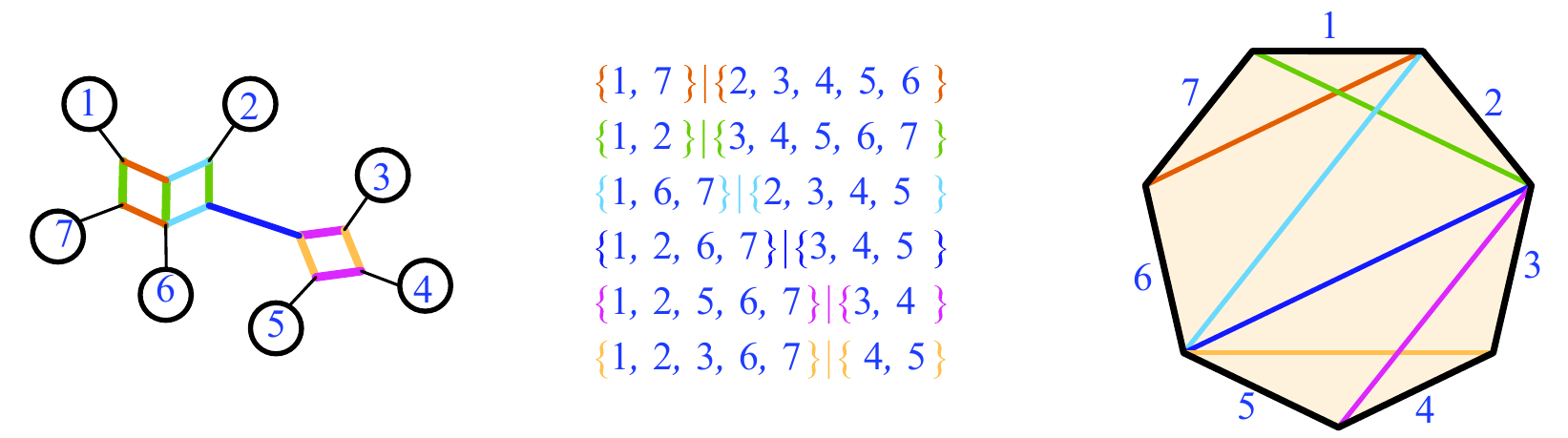}}
\caption{ A split system, in the center, with its circular split network, and its polygonal representation.The trivial splits are assumed to be included. This network is externally refined, so no bridges can be added.}
\label{f:network}
\end{figure}

In practice, the construction of a  split network from an arbitrary dissimilarity matrix tends to yield an overcomplicated system. Instead, a common heuristic is to choose a circular ordering compatible with most of the splits in order to produce a \emph{circular split network}, adding the remaining splits at a
later time \cite{hs}.  Bryant and Moulton introduced the \emph{neighbor-net} algorithm  in 2004, equipped to construct such a circular split network from dissimilarity matrices satisfying the Kalamanson condition \cite{bm}. Neighbor-net is a greedy algorithm, but the circular split networks output by the algorithm are informative for exploring conflicting signals in data.

Applications of our new polytopes begin with linear programming. We have discovered a series of polytopes whose vertices are phylogenetic networks, filtering between the Balanced Minimal Evolution (BME) polytopes and the Symmetric Traveling Salesman polytopes (STSP). We are using combinations of the simplex method, branch and bound, and cutting planes on these polytopes in order to choose precisely which network best fits the data. The discovery of large families of facets in every dimension allows us to attack the BME problem.

\section{Trees and Networks}

We begin with the set $[n] = \{1,2,\dots,n\}$ used to number $n$ taxa, species, or genes.  A \textit{split} is a partition of $[n]$ into two parts, denoted $A|B$. A split is called \emph{nontrivial} when both parts $A$ and $B$ have more than one element. A \textit{split system} is a set of splits. We will always assume that our split systems contain all the trivial splits, with one part having a single element. Split systems may be drawn as connected graphs called \textit{split networks} in which the degree-one nodes  (leaves) are given the labels $1\dots n,$ and nodes of higher degree have no label. The interior (un-leafed) edges are drawn in sets of parallel edges which together represent a split: by cutting such a set of edges the graph separates into the parts of the split (Figure~\ref{f:network}). If a split is represented by a single edge in the network, we call it a \emph{bridge}. We say a split system $s'$ \emph{refines}  $s$ when $s' \supset s.$ In the network picture some splits of $s'$ are collapsed (the parallel edges are assigned length zero) to achieve $s.$

A \textit{phylogenetic tree} is a circular split network whose splits form a tree. That is, a network for which each interior edge is a bridge: cutting it disconnects the graph into two smaller graphs, called \textit{clades}.  A \textit{circular} or \textit{exterior planar} split network is one whose graph may be drawn on the plane without edge crossings, with all leaves on the exterior.

There exists a \emph{dual polygonal  representation} of a circular split network:  Given a circular split system with a circular ordering $c$ of the species, consider a regular $n$-gon, with the edges cyclically labeled according to $c$.  For each split, draw a diagonal partitioning the appropriate edges (Figure
\ref{f:network}). Note that the diagonals which represent bridges in the split network  are \emph{noncrossing}, that is the diagonal does not intersect any other diagonal. A \textit{fully reticulated} split network is one with no  bridges.

\begin{defn}
An \emph{externally refined split network} $s$ is such that there is no split network $s'$ both refining $s$ and possessing more bridges than $s$.\footnote{In the polygonal representation, there are no non-crossing diagonals that can be added.}  Note that an externally refined network can be a refinement of another externally refined network. An \emph{externally refined phylogenetic tree} has all internal nodes of degree three (usually called a \textit{binary} tree).
 \end{defn}

 Another generalization of an unrooted phylogenetic tree is an (unrooted) \emph{phylogenetic network}. The following definitions are from \cite{Gambette2017}: A phylogenetic network is a simple connected graph with exactly $n$ labeled nodes of degree one, all other unlabeled nodes of degree at least three, and every cycle of length at least four.\footnote{Cycles here are simple cycles, with no repeated nodes other than the start.} If every edge is part of at most one cycle, then the network is called \emph{1-nested}. If every node is part of at most one cycle, the network is called \emph{level-1}. If that is true and  all the unlabeled non-leaf nodes also have degree three, then the network is called \emph{binary level-1}. Level-1 networks as a set contain the level-0 networks, which are the phylogenetic trees. Notice that a phylogenetic tree is both a split network and a level-0 phylogenetic network.

 A \emph{minimal cut} of a phylogenetic network is a subset of the edges which, when removed, leaves two connected components. The edge set is minimal in the sense that no more edges are removed than is necessary for the disconnection. The split displayed by such a cut is the two sets of leaves of the two connected components.  A split is \emph{consistent} with a phylogenetic network if there is a minimal cut displaying that split. The system of all such splits for a phylogenetic network $N$ is called $\Sigma(N).$

 In Figure~\ref{my_splits_ex_SN}, we show a binary level-1 network $N$, and the associated maximal split network $\Sigma(N)$. There is a close relationship between level-1 (and thus 1-nested) networks and circular split networks. In \cite{Gambette2017}, it is shown that a split network $s$ is circular if and
only if there exists an unrooted 1-nested network $N$ such that $s \subset \Sigma(N).$ For instance the split network $s$ in Figure~\ref{f:network} has splits a subset of those in $\Sigma(N)$ seen in Figure~\ref{my_splits_ex_SN}.

 \begin{figure}[h]
\centering{\includegraphics[width=\textwidth]{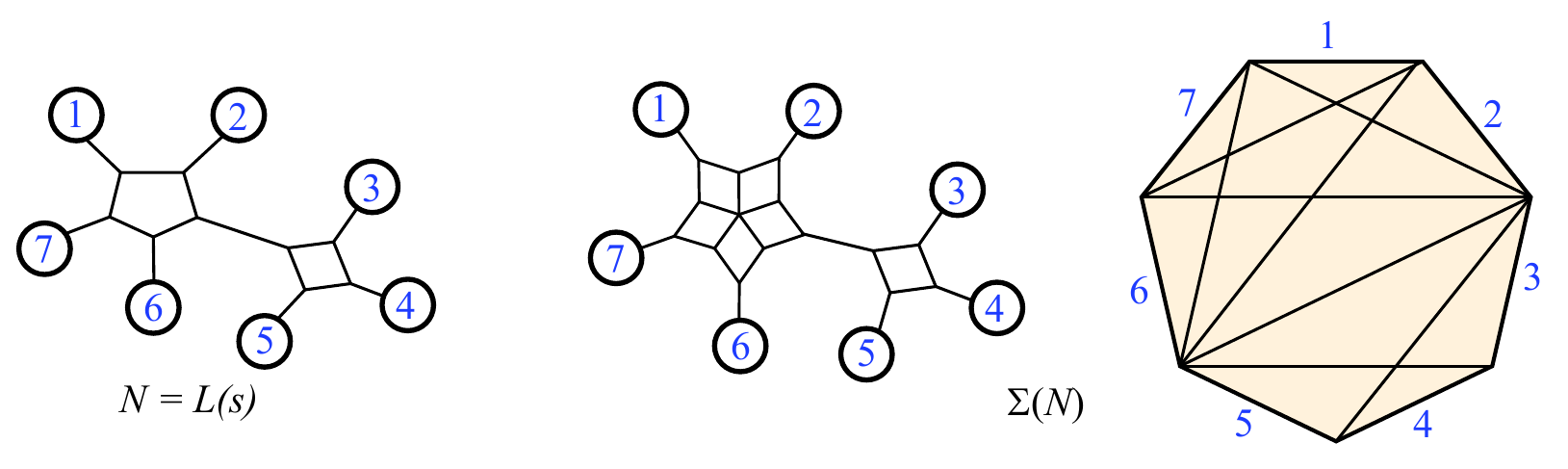}}
\caption{ A level-1 phylogenetic network $N$ with its associated maximal circular split system $\Sigma(N)$, shown both as a network and polygonal representation. Here, $N$ is the the image $L(s)$ for $s$ the split network from Figure~\ref{f:network}.}
\label{my_splits_ex_SN}
\end{figure}
\begin{defn}
If $s$ is a circular split system, then there is a simple way to associate to $s$ a specific 1-nested phylogenetic network denoted as $L(s)$.  Construct this network $L(s)$ as follows:  begin with a split network diagram  of $s$ and  consider the diagram as a drawing of its underlying planar graph. Recall that the leaves are on the exterior.  First, delete all the edges that are not adjacent to the exterior of that graph.  And second, smooth away any resulting degree-2 nodes.
\end{defn}
As an example, Figure~\ref{f:network} shows $s$, whereas $L(s)$ is in Figure~\ref{my_splits_ex_SN}. We see that $L(s)$ displays all the splits of $s$, and has the same bridges as $s$.  Note that in \cite{Gambette2017} the authors define a similar function called $N$. Their function is equivalent to ours, if the definition in \cite{Gambette2017} is modified so as to not depend on $k$-marguerites.

\begin{defn}
A \emph{weighting} (or \emph{metric}) of a split system $s$ is a function $w:s\to\mathbb{R}_{\ge 0}$.  Given such a weighting, we can derive a metric $d_s$ on $[n],$ via
$$d_s(i,j) = \sum_{i\in A, j\in B} w(A|B)\,,$$
where the sum is over all splits of $s$ with $i$ in one part and $j$ in the other. The metric is often referred to as the distance vector $\mathbf{d_s}.$ We define the total weight of the network to be the sum of all the weights: $W(s) = \sum_{A|B \in s} w(A|B)$.
\end{defn}

Often in practice each split is assigned a positive weight, since when splits are assigned zero weights, the system can be equated to the same system minus those splits. This latter equating is important  since it is necessary not just to understand individual network structures, but  the relationships between them.  Billera, Holmes, and Vogtmann laid the foundation for this process by constructing a geometric space $\BHV{n}$ of such metric trees with $n$ leaves \cite{bhv}.  This space is contractible (in fact, a cone),
formed by gluing $(2n-5)!!$ orthants $\R_{\geq 0}^{n-3}$, one for each type of labeled binary tree. As the weights go to zero, we get degenerate trees on the boundaries of the orthants, where two boundary faces are identified when they contain the same degenerate trees. The space $\BHV{n}$ is not a manifold but a cone over a relatively singular simplicial complex.

Devadoss and Petti recently constructed a \emph{geometric space} $\sn{n}$ of metric circular split networks with $n$ labeled leaves \cite{dev-petti}: A circular split network with $n$ leaves has at most ${n \choose 2} -n = n(n-3)/2$ splits compatible with a specific ordering. From a polygonal perspective, this is the maximal number of diagonals on an $n$-gon.  Thus, the vector of interior edge lengths $(l_1, \dots, l_{n(n-3)/2} )$ specifies a point in the orthant $[0, \infty)^{n(n-3)/2}$, defining coordinate charts for the space of such networks.  That is, to each  point in this orthant, we associate the unique network which is combinatorially equivalent but with differing edge lengths, specified by the coordinates of that point.

The space of networks is assembled from $(n-1)!/2$ orthants, each of which corresponds to a unique circular ordering of the $n$ species, up to rotation and reflection.  As the interior edge lengths go to zero, we get degenerate networks representing the common boundary faces of these orthants. The spaces CSN(4) and BHV(4) are shown in Figure~\ref{f:duality}.
The link of the origin $\n{n}$ of $\sn{n}$ is the union of the set of points in each orthant with internal edge lengths of networks  that sum to 1.

\begin{thm} \cite{dev-petti}
The space $\n{n}$ is a connected simplicial complex of dimension  ${n \choose 2}-n-1$, with one $k$-simplex for every labeled $n$-gon with $k+1$ diagonals.
\end{thm}

Unlike the link of the origin of $\BHV{n}$, whose homotopy structure has been known for several decades, little is known about the topology of $\n{n}$.
A partial key lies with the novel gluing of the chambers of $\n{n}$: we have recently shown that two chambers of $\n{n}$ can intersect along a face of \emph{at most dimension} $\binom{n-2}{2} - 1.$

\section{Polytopes}

We now turn to a wonderful family of polytopes which organize these split networks.  Our new polytope collection is nested between the STSP and BME polytopes.

\begin{defn}
For each circular ordering $c$ of $[n]$, the \textit{incidence vector} ${\mathbf x}(c)$  has ${n \choose 2}$ components. The component $x_{ij}= 1$ if  $i$ and $j$ are adjacent in $c$, and $x_{ij}= 0$ if not. The $n(n-3)/2$-dimensional \emph{Symmetric Traveling Salesman Polytope} STSP($n$) is the convex hull of these vectors.
\end{defn}

\noindent
A  circular ordering $c$ is \textit{consistent} with a circular split system $s$ if a planar network of $s$ may be drawn so its leaves lie on the exterior in the order given by $c$. Examples are calculated in Figure~\ref{f:filter} and their associated STSP(4) is shown in Figure~\ref{f:duality}.

The Balanced Minimal Evolution Polytope BME($n$) was first studied in 2008 \cite{Rudy2008}. We have recently found a simple description as follows \cite{forcey2015facets}:

\begin{defn}
For each given phylogenetic tree $t$ with $n$ leaves, the \textit{vertex vector} $\mathbf{x}(t)$ has ${n \choose 2}$ components $x_{ij}(t) = 2^{n-3-b_{ij}}$ where $b_{ij}$ is the number of nontrivial bridges on the path from leaf $i$ to a different leaf $j.$
The convex hull of all the  $(2n-5)!!$ vertex vectors (for all binary trees $t$ with $n$ leaves), is the polytope BME$(n),$  of dimension ${n \choose 2}-n.$
\end{defn}

We define new families of polytopes by assigning  vectors to each externally refined circular split network $s$, and thus to each binary level-1 phylogenetic network.

\begin{defn}
The vector ${\mathbf x}(s)$ is defined to have components ${\mathbf x}_{ij}(s)$ for each unordered pair of leaves $i,j \in [n]$ as follows:
\begin{equation} \label{e:bmenkvert} {\mathbf x}_{ij}(s) = \begin{cases} 2^{k-b_{ij}} & \text{if there exists $c$ consistent with $s$; with $i,j$  adjacent in $c$,}\\ 0 & \text{otherwise.} \end{cases} \end{equation}
Here, $k$ is the number of bridges in $s$ and $b_{ij}$ is the number of bridges crossed on any path from $i$ to $j$.
\end{defn}

Note that though both variables in the formula depend on $s$, they are determined entirely by $L(s).$   Thus two externally refined circular split systems with the same associated binary  level-1 network will have the same vector ${\mathbf x}.$

\begin{defn}
The convex hull of all the vectors ${\mathbf x}(s)$ for $s$ an externally refined circular network with $n$ leaves and $k$ bridges is the \emph{level-1 network polytope} BME($n,k$).
\end{defn}

\begin{thm} \label{cool}
The vertices of BME($n,k$) are the vectors ${\mathbf x}(s)$ corresponding to the distinct binary level-1 networks $L(s)$.  That is for each externally refined circular network $s$, with $n$ leaves and $k$ bridges, we get an extreme point of the polytope (but it is determined only by $L(s).$)  The dimension of BME($n,k$) is ${n \choose 2}-n.$
\end{thm}

\begin{proof}
To show that the vectors thus calculated are extreme in their convex hull, we use the fact that each is the sum of the vertices of the STSP  which correspond to the circular orderings consistent with that network. Let ${\mathbf d}_s$ be the distance vector whose $i,j$ component is the path length between those leaves on $s.$ Then the linear functional corresponding to the distance vector is simultaneously minimized at each of these (consistent circular ordering) vertices of the STSP, and thus uniquely minimized at the vertex  ${\mathbf x}(s)$.

For the dimension we use the fact that each polytope BME($n,k$) is nested between STSP($n$) and BME($n$). We also have the following: for each leaf $j=1,\dots,n$ these vertices ${\mathbf x}(s)$ satisfy $\sum_{i=1, i\ne j}^n x_{ij}= 2^{k+1}\, ,$ where $k$ is the number of bridges (non-crossing diagonals) in the diagram. More details of proofs are found in \cite{bmenk1}.
\end{proof}

\begin{cor}
Restricting BME($n,k$) to the phylogenetic trees,  where $k=n-3$, recovers the polytopes BME($n$). Restricting BME($n,k$) to the fully reticulated networks, where $k=0$, turns out to recover STSP($n$).
\end{cor}

\begin{thm}
Every $n$ leaved 1-nested network $S$ with $m$ bridges corresponds to a face $F(S)$ of each BME($n,k$) polytope for $k\le m.$  The vertices of the face $F(S)$ are all the binary level-1 $k$-bridge networks $S'$ whose splits refine those of $S$; that is, $\Sigma(S)\subset\Sigma(S')$.
\end{thm}
For examples see Figure~\ref{f:filter}, where shaded subfaces are labeled by the corresponding network with additional bridges. Those subfaces link up to make an interesting complex.  The topology of these complexes in general is an open question.
\begin{proof} Without loss of generality we choose a split network $s$ which has the exterior form of $S,$ that is $L(s) = S.$
Let $s$ be weighted by assigning the value of 1 to each split. Then $W(s)$ is the total number of splits in $s$. Let ${\mathbf d}_s$ be the distance vector derived from that weighting so that the $i,j$ component of ${\mathbf d}_s$ is the number of splits between those leaves on $s.$ We see that the dot product ${\mathbf x}(s')\cdot{\mathbf d}_s$ is minimized simultaneously at each of the $k$-bridge networks $s'$ which externally refine $s$. In fact we have that the following inequality
$$
{\mathbf x}(s')\cdot{\mathbf d}_s \ge 2^{k+1}W(s)
$$
holds for all $k$-bridge externally refined networks $s'$, and is an equality precisely when $s'$ refines $s.$
The reason is that ${\mathbf d}_s$ is equivalent to a distance vector derived from $s'$, where the splits are given weight $= 1$ if they are also in $s$, and weight $= 0$ if not. Thus the dot product will equal $2^{k+1}W(s)$, and that value will be a minimum.
\end{proof}

We have found that all of the known facets of BME($n$) have analogues in BME($n,k$) for small $n.$ For large $n$ we also see the following:

\begin{thm}
Any split $A|B$ of $[n]$ with $|A|>1$, $|B|>1$ corresponds to a face of BME($n,k$), for all $n,k$ with $ k\le n-3$.  The vertices of that face are the binary level-1
 networks which display the split $A|B.$ Furthermore, if $|A|>2$, $|B|>2$, the face is a facet of the polytope.  The face inequality is:
 $$\sum_{i,j\in A} x_{ij} \le (|A|-1)2^k.$$
 \end{thm}

\begin{proof}
First we show that the collection of vertices corresponding to networks displaying a split $A|B$ of $[n]$ obey our linear equality, and that all other vertices obey the corresponding inequality. Then we use the fact that for any polytope, its scaling by $m$ is of the same dimension. Equivalently, taking sets of $m$ vectors all from the same facet of a polytope, the vector sums of $m$ such vectors will all lie in an affine space of the same dimension as that facet. (Thus any subset of those sums of $m$ vectors each will have a convex hull of smaller or equal dimension than the original facet.)

We have that a given split, with both parts larger than two, corresponds to a facet of STSP($n$) = BME($n,0$) and also to a facet of BME($n,n-3$) = BME($n$). The vertices of the proposed split-facet $F_A(n,k)$ are each formed by vector summing exactly two vertices of $F_A(n,k-1).$

Now we see that from $k=0$ to $k=n-3$ at each step we find that the facet $F_A(n,k)$ cannot be of greater dimension than the  facet $F_A(n,k-1).$ Since the dimension cannot increase at any step, and it has the same value for $k=0$ and for $k=n-3$,  then it must remain constant for each $k$ at ${n \choose 2} - n - 1.$
\end{proof}

We know the case of splits with one part of size two corresponds to a facet when $k=0$, the STSP($n$), but not when $k=n-3$, the BME($n$). It is an open question for which other $(n,k)$ pairs do the splits of size two correspond to facets of BME($n,k$.) We conjecture this for all $k<n-3,$ but we can only report the positive result for $n=5.$

\begin{figure}[h!]
\begin{center}
\includegraphics[width=\textwidth]{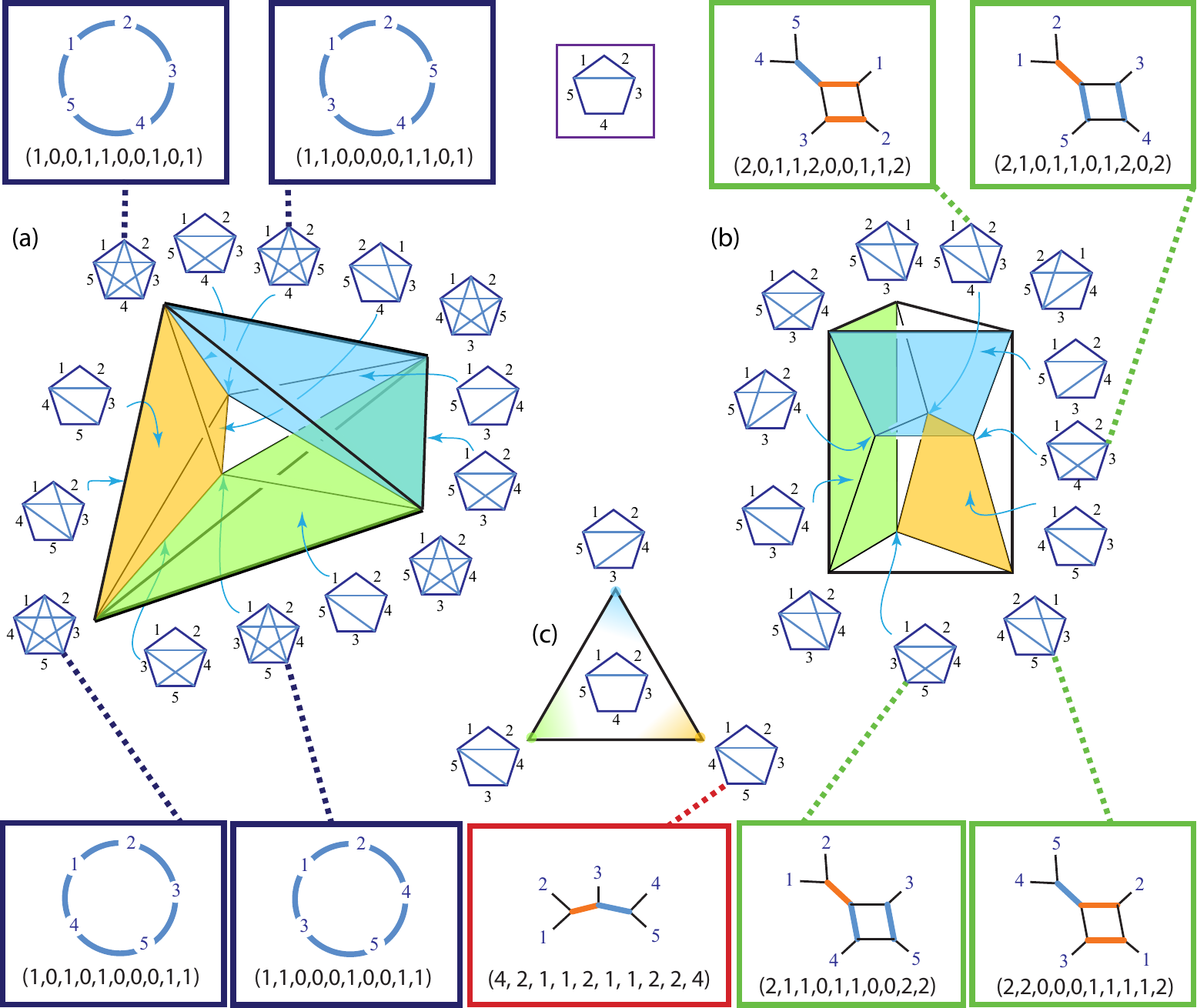}
\end{center}
\caption {A facet (a) in STSP(5) = BME(5,0) , a facet (b) in BME(5,1), and a face (c) in BME(5) = BME(5,2).  All three correspond to the same split, pictured at the top.}
\label{f:filter}
\end{figure}

Figure~\ref{f:filter}(a) shows a subtour elimination facet  of STSP(5) = BME(5,0) corresponding to the split $s=\{\{1,2\},\{3,4,5\}\}.$ Split networks label subfaces of this facet. Figure~\ref{f:filter}(b) shows the corresponding split facet of BME(5,1) in which the vertices are nine networks with a single bridge that each refine $s$. Figure~\ref{f:filter}(c) shows the face of BME(5) corresponding to the same split. Summing either of the horizontal or  vertical pairs of vectors shown in (a) gives the four vectors shown in (b).  Summing all four vectors in (a) gives the vector shown in (c).


\section{Facets of BME($n$) and Duality with BHV${}_n$}

Until our recent work, only certain of the low-dimensional faces of BME($n$) were known, namely the \textit{clade faces} as described in \cite{Rudy}. Now we have discovered exponentially large collections of (maximum dimensional) facets for all $n$ \cite{forcey2015facets}.  In the  list below,  we review our new facets and point out how they often fit yet another larger pattern. Namely, faces of the BME($n$) polytope often arise as vertex collections corresponding to binary trees which  display any maximal compatible subset of the splits of a given circular network.
These collections of trees have a biological application: The binary trees which  display any maximal compatible subset of the splits of a given network are known as the fully resolved trees allowed by the network. The number of these for a given network on $n$ leaves, divided by the number of binary phylogenetic trees possible with $n$ leaves, gives us the cladistic information content (CIC) of the tree as defined in \cite{Gauthier2007}.

\begin{enumerate}
\item Any split of $[n]$ with both parts larger than 3 corresponds to a facet of BME$(n)$, with vertices all the trees displaying that split.  These vertices constitute the CIC collection  of that split. The number of these facets grows like $2^n.$

\item A \emph{cherry} is a clade with two leaves. For each intersecting pair of cherries $\{a,b\},\{b,c\}$, there is a facet of BME$(n)$ whose vertices correspond to trees having either cherry. The facet inequality is $x_{ab}+x_{bc} -x_{ac}\le 2^{n-3}$. These vertices constitute the CIC collection corresponding to a certain network in CSN${}_n$.

\item For each pair of leaves $\{i,j\}$, the caterpillar trees with that pair fixed at opposite ends constitute the vertices of a facet. These bound BME$(n)$ from below: $x_{ij} \ge 1$.
\end{enumerate}

Of course, the number of facets of BME($n$) grows much more quickly with $n$ than the number of facets we know. But these faces are still interesting and useful.  One reason is that we have discovered that the facet structure of BME($n$) has a direct relationship to the quotient structure of BHV${}_n$. Let $\mathcal{L}$(BME($n$)) be the poset of faces of the BME polytopes, and let $\mathcal{L}$(BHV${}_n$) be the poset of cells of BHV${}_n$, both ordered by inclusion.

\begin{thm}
There exists a poset injection:   $f:\mathcal{L}$(BHV${}_n$) $\to$  $\mathcal{L}($BME$(n)^{\Delta})$. In particular
the $(2n-5)!!$ top-dimensional cells of BHV${}_n$ map to the $(2n-5)!!$ vertices of BME($n$).
\end{thm}

Specifically,  any (non-binary) phylogenetic
tree $t$ corresponds to a face $f(t)$ of BME$(n)$ which contains the vertices corresponding to binary trees that refine $t$. For an example, consider the triangular face in Figure~\ref{f:filter}(c). Further, these faces are ordered by inclusion exactly  opposite the inclusion in BHV${}_n.$
For each (non-binary) $t$, there is a distance vector
${\mathbf d}_t$ for which the product ${\mathbf d}_t\cdot {\mathbf
x}(t')$ is minimized simultaneously by the set of \emph{binary}
phylogenetic trees $t'$ which refine $t.$ In particular, for any
tree $t'$, we have
$$\sum_{i<j} \ d_{ij}(t) \ x_{ij}(t') \ \ge \ 2^{(n-2)}|E(t)|\, ,$$
where $E(t)$ is the set of edges of $t.$ This inequality is precisely an equality if and only if the tree $t'$
is a refinement of $t.$ Thus the distance vector is the normal to
the face. The map thus described is an injective poset map from BHV${}_n$ to faces in BME($n$), ordered by reverse inclusion (the polar ordering).  Figure~\ref{f:duality} shows an illustration when $n=4.$

\begin{figure}[h]
\begin{center}
\includegraphics[width=\textwidth]{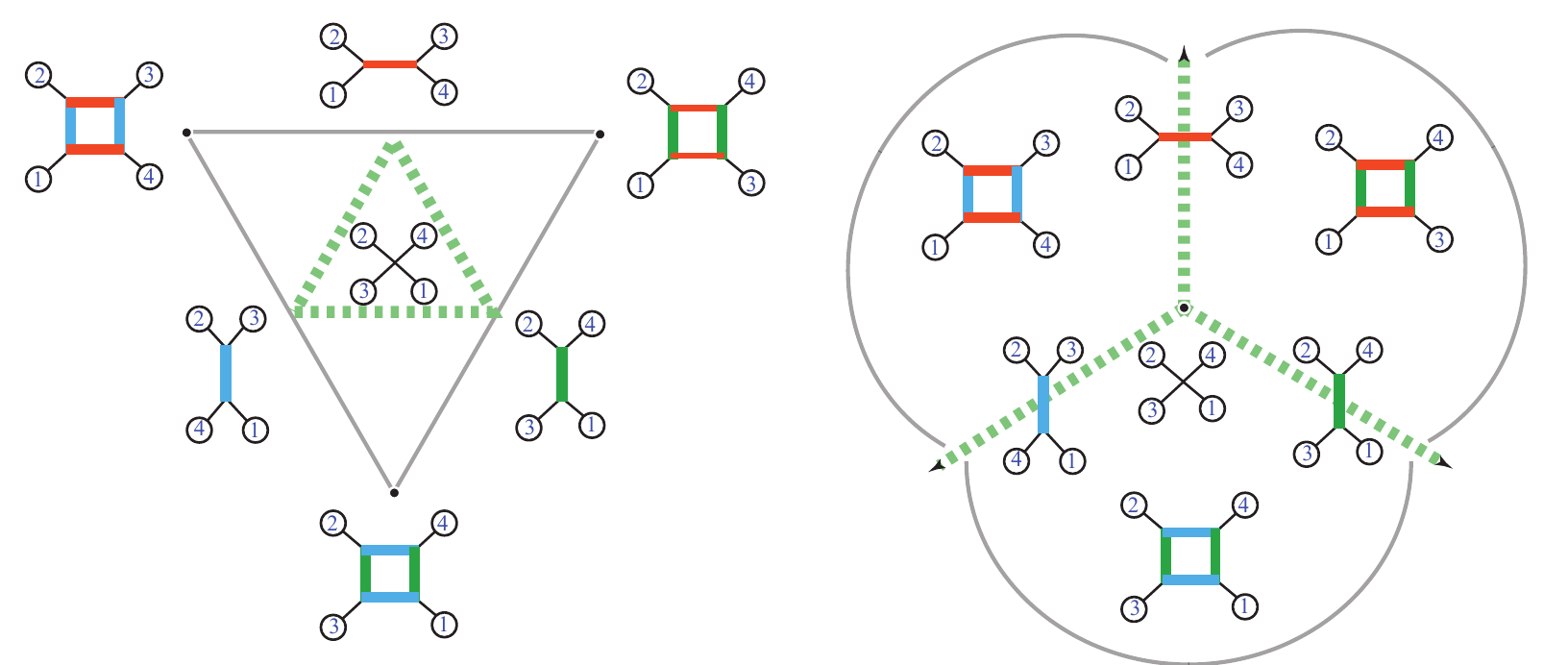}
\end{center}
\caption {STSP(4) = BME(4,0) on the left is dual to CSN(4) on the right. BME(4) = BME(4,1) is dual to BHV(4); these are the restrictions to trees.}
\label{f:duality}
\end{figure}

This discovery allows us to define a projection from Kapranov's
\emph{permutoassociahedron} $\K P_n$ to a complex of the BME$(n)$ polytope,taking faces to faces and preserving the partial order of faces. This polytope $\K P_n$ can be viewed as a polyhedral and combinatorial analog of both tree spaces $\BHV{n}$ and the real moduli space of curves $\M{n}$; see \cite{dev1} for details.  The polytope $\K P_n$ blends
two classical polyotopes:   the \emph{associahedron}  (measuring tree structures) and the permutohedron (measuring permutations of leaves). The details of the map from $\K P_n$ to BME$(n)$ are given in \cite{splito}.

We have extended this discovery of duality to the space of networks CSN${}_n.$ In the following, the projection $f$ defined on networks is the same map $f$ as just described when restricted to trees.

\begin{thm}
The polytope STSP($n$) has a complex of subfaces which is the dual image of a projection $f$ from CSN${}_n.$  The $(n-1)!/2$ orthants of CSN${}_n$ map to the vertices of STSP($n$), and the networks with a single split map to the
the sub-tour elimination facets of STSP($n$).
\end{thm}

\noindent Equivalently, we could say $f$ projects network space onto a complex in the \textit{fan} of the STSP. This theorem is illustrated in Figures~\ref{f:filter} and~\ref{f:duality}. A network $s$  maps to the face $f(s)$ of STSP($n$) whose vertices are circular orderings consistent with $s.$

\section{Applications}

The \emph{balanced minimum evolution} method
reconstructs  phylogenetic trees by minimizing the total tree
length.  This method is \emph{statistically consistent} in that as
the dissimilarity matrix approaches the zero-noise accuracy of an additive
metric, the BME output approaches that tree's true topology.   The BME tree for an arbitrary positive vector $\mathbf{d}$ is the binary tree $t$ that minimizes $\mathbf{d}\cdot\mathbf{x}(t)$ for all binary trees with $n$ leaves,\footnote{If the dissimilarity matrix comes from a tree matrix, this dot
product is simply a scaled sum of all the edge lengths, a sum which is known as the \emph{tree length}.} this \emph{objective value} dot product being the least variance estimate of tree-length \cite{Despernew}. With our discovery of large collections of facet inequalities comes the opportunity to infer the BME tree directly by  linear optimization over a polytope. The problem of course is that we know only a few of the many facet inequalities.

If we take all the split-faces of the BME polytope, including the
cherries and caterpillar facets, the resulting intersection of half-spaces becomes a bounded polytope, appearing inside the ${n \choose 2}$-cube. We can then add the intersecting-cherry facets and other split-network facets and note
that this new polytope envelopes the BME polytope by restricting to
a known subset of the latter polytope's facets. The advantage of
considering this \emph{relaxed BME polytope} is that it shares many
of the same vertices, occasionally allowing linear programming over the
relaxation to work just as well as linear programming over BME$(n)$.
Our algorithm, {\sc PolySplit}, does exactly this and is outlined in \cite{forcey-book}.

We have shown that for an $n$-leaved binary phylogenetic tree, if the number of cherries is at least $n/4$, the tree represents a vertex in the BME polytope  is also a vertex of our relaxation \cite{splito}. For $n\le 11$, the
tree represents a vertex of the relaxation regardless of the number
of cherries.  For BME(5) = BME(5,2), for instance, there are 15 vertices and 52 facets. The relaxed version in 5D has all the same facets minus 12 of them, so 40, and it turns out to possess exactly the same vertices plus 12 new ones. In general, the vertices associated to binary phylogenetic trees for any $n$ all lie on the common boundary of several
facets of the relaxation, which are also facets of the BME polytope. Therefore we believe the relaxation will prove to be useful as a proxy for solving the BME problem via linear optimization.

%
%

\bibliographystyle{amsplain}
\bibliography{phylogenetics}{}

\end{document}